\documentclass[12pt]{amsart}
\usepackage{amssymb}
\usepackage{colordvi}
\usepackage{graphicx}
\usepackage{mathrsfs}   
\usepackage{vmargin}
\setpapersize{USletter}
\setmargrb{2.5cm}{2.5cm}{2.5cm}{2.5cm} 

\numberwithin{equation}{section}
\newtheorem{theorem}{Theorem}[section]
\newtheorem{proposition}[theorem]{Proposition}
\newtheorem{lemma}[theorem]{Lemma}

\newtheorem{question}[theorem]{Question}
\newtheorem{conjecture}[theorem]{Conjecture}
\theoremstyle{remark}
\newtheorem{definition}[theorem]{Definition}
\newtheorem{example}[theorem]{Example}
\newtheorem{remark}[theorem]{Remark}
\newcounter{FNC}[page]
\def\fauxfootnote#1{{\addtocounter{FNC}{2}\Magenta{$^\fnsymbol{FNC}$}%
     \let\thefootnote\relax\footnotetext{\Magenta{$^\fnsymbol{FNC}$#1}}}}
\def\cprime{$'$}

\newcommand{\longtwoheadrightarrow}{\ensuremath{\relbar\joinrel\twoheadrightarrow}}

\newcommand{\CC}{{\mathbb C}}
\newcommand{\RR}{{\mathbb R}}
\newcommand{\NN}{{\mathbb N}}
\newcommand{\PP}{{\mathbb P}}
\newcommand{\TT}{{\mathbb T}}
\newcommand{\UU}{{\mathbb U}}
\newcommand{\ZZ}{{\mathbb Z}}

\newcommand{\bx}{{\bf x}}
\newcommand{\calV}{{\mathcal V}}
\newcommand{\calU}{{\mathcal U}}

\newcommand{\scrE}{\mathscr{E}}
\newcommand{\scrL}{\mathscr{L}}
\newcommand{\scrA}{\mathscr{A}}
\newcommand{\coscrA}{{\it co}\mathscr{A}}

\DeclareMathOperator{\Log}{{\rm Log}}
\DeclareMathOperator{\Arg}{{\rm Arg}}
\DeclareMathOperator{\Hom}{{\rm Hom}}

\newcommand{\defcolor}[1]{\Blue{#1}}
\newcommand{\demph}[1]{\defcolor{{\sl #1}}}

\title{Describing Amoebas}
\author{Mounir Nisse}
\address{Mounir Nisse\\
   Department of Mathematics\\
   Xiamen University Malaysia\\
  }
\email{mounir.nisse@gmail.com}
\author{Frank Sottile}
\address{Frank Sottile\\
         Department of Mathematics\\
         Texas A\&M University\\
         College Station\\
         Texas \ 77843\\
         USA}
 \email{sottile@math.tamu.edu}
\urladdr{www.math.tamu.edu/\~{}sottile}
\thanks{Research of Sottile supported in part by NSF grant DMS-1501370.}
\thanks{Both authors thank the Institute Mittag-Leffler, where this was conceived and written}
\begin{document}

\begin{abstract}
 An amoeba is the image of a subvariety of an algebraic torus under the logarithmic moment map.
 We consider some qualitative aspects of amoebas, establishing results and posing problems for further study.
 These problems include determining the dimension of an amoeba, describing an amoeba as a semi-algebraic set, and
 identifying varieties whose amoebas are a finite intersection of amoebas of hypersurfaces.
 We show that an amoeba which is not of full dimension is not such a finite intersection if its variety is nondegenerate and
 we describe amoebas of lines as explicit semi-algebraic sets.
\end{abstract}

\maketitle

\section{Introduction}

Hilbert's basis theorem implies that an algebraic variety is the intersection of finitely many hypersurfaces.
A tropical variety is the intersection of finitely many tropical hypersurfaces~\cite{BJSST,HT}.
These results not only provide a finite description of classical and tropical varieties, but they 
are important algorithmically, for they allow these varieties to be represented and manipulated on a computer.
Amoebas and coamoebas are other objects from tropical geometry that are intermediate between classical and tropical
varieties, but less is known about how they may be represented.

The amoeba of a subvariety $V$ of a torus $(\CC^\times)^n$ is its image under the logarithmic moment map
$(\CC^\times)^n\to\RR^n$~\cite[Ch.~6]{GKZ}. 
The coamoeba is the set of arguments in $V$.
An early study of amoebas~\cite{Th02} discussed their computation.
Purbhoo~\cite{Purbhoo} showed that the amoeba of a variety $V$ is the intersection of amoebas of all hypersurfaces
containing $V$ and that points in the complement of its amoeba are witnessed by certain lopsided polynomials in 
its ideal.
This led to further work on approximating amoebas~\cite{TdW}.
Schroeter and de Wolff~\cite{SdW}, and then Nisse~\cite{Nisse} showed that the amoeba of zero-dimensional variety is an
intersection of finitely many hypersurface amoebas.
More interestingly, Goucha and Gouveia~\cite{GG} showed that the amoeba of the variety of $m\times n$ matrices of rank less than
$\min\{m,n\}$ is the (finite) intersection of the hypersurface amoebas given by the maximal minors. 

A subvariety of the torus has a finite amoeba basis if its amoeba is the intersection of finitely many hypersurface  
amoebas. 
This property is preserved under finite union. 
We show that a complete intersection of polynomials whose Newton polytopes are affinely independent has a finite amoeba
basis.
Despite this and the results in~\cite{GG,Nisse,SdW}, we expect it is rare for a variety to have a finite amoeba basis.
Indeed, we show that if a subvariety $V\subset(\CC^\times)^n$ has an amoeba of dimension less than
$n$ and a finite amoeba basis, then each component of $V$ lies in some translated subtorus.

Determining which varieties admit a finite amoeba basis will require a better understanding of  amoebas, for which we suggest two
problems.
In  Section~\ref{S:dimension}, we identify a structure of a variety $V$ which implies that it has dimension
smaller dimension than expected, $\min\{n,2\dim_\CC V\}$.
We originally conjectured that this condition was necessary, which was
subsequently proven by Draisma, Rau, and Yuen~\cite{DRY}.
When the amoeba of $V$ has the minimal possible dimension, we show that $V$ is a single orbit of a subtorus.
It remains an open problem to develop methods to recognize when a variety has this structure.
The second problem asks for a description of the coamoeba and algebraic amoeba (the projection of $V$ to $\RR_>^n$) of a
variety $V$ as semi-algebraic sets. 
Both questions, dimension and description as semi-algebraic sets are open and interesting for linear subspaces.
We exhibit such a description for amoebas of lines.

We give some background in Section~\ref{S:background}.
In Section~\ref{S:SA} we observe that coamoebas and algebraic amoebas are semi-algebraic sets and describe the algebraic
amoeba of a line as a semi-algebraic set.
Section~\ref{S:dimension} considers the problem of determining the dimension of an amoeba, and in Section~\ref{S:not_finite} we study when a
variety has a finite amoeba basis. 

We thank Timo de Wolff, Avgust Tsikh, Ilya Tyomkin, Jan Draisma, and Jo\~{a}o
Gouveia for stimulating conversations.
We also thank the Institute Mittag-Leffler for its hospitality during the writing of this manuscript.

\section{Amoebas and coamoebas}\label{S:background}

A point $z$ in  the group \defcolor{$\CC^\times$} of nonzero complex numbers
is determined by its absolute value \defcolor{$|z|$} and its argument \defcolor{$\arg(z)$}.
These are group homomorphisms that identify $\CC^\times$  with the product $\RR_>\times\UU$,
where $\defcolor{\RR_>}\subset\CC^\times$ is its subgroup of positive real numbers and $\defcolor{\UU}\subset\CC^\times$ is
its subgroup of unit complex numbers. 
The decomposition $\CC^\times\simeq\RR_>\times\UU$ induces 
a decomposition of the complex torus $(\CC^\times)^n\simeq\RR_>^n\times\UU^n$.
We consider the projections
$|\cdot|\colon(\CC^\times)^n\to\RR_>^n$ and $\Arg\colon(\CC^\times)^n\to\UU^n$ to each factor.
Let $\defcolor{\Log}\colon(\CC^\times)^n\to\RR^n$ be the composition of the projection $|\cdot|$ with the coordinatewise
logarithm.
This is called the logarithmic moment map.

A subvariety  $V\subset(\CC^\times)^n$ of the torus is the set of zeroes of finitely many Laurent polynomials.
It is a \demph{hypersurface} when it is given by a single polynomial, and a \demph{complete intersection}
if it is given by $r=n{-}\dim V$ polynomials.
The \demph{amoeba} $\defcolor{\scrA(V)}\subset\RR^n$ of a subvariety $V\subset(\CC^\times)^n$ of the
torus is its image under $\Log$ and its \demph{coamoeba}  $\defcolor{\coscrA(V)}\subset \UU^n$ is its image under $\Arg$. 
Gelfand, Kapranov, and Zelevinsky defined amoebas~\cite[Ch.~6]{GKZ} and coamoebas first appeared in a 2004 lecture
of Passare. 
The \demph{algebraic amoeba} $\defcolor{|V|}\subset\RR_>^n$ of a subvariety $V$ is its image under
the projection $(\CC^\times)^n\to \RR_>^n$.
Because $\RR_>^n\simeq \RR^n$ under the logarithm and exponential maps, $|V|\simeq\scrA(V)$ as analytic subsets of their
respective spaces.

As the map $\CC^\times\to\RR_>$ is proper, the maps $(\CC^\times)^n\to\RR_>^n$ and
$(\CC^\times)^n\to\RR^n$ are proper, and therefore algebraic amoebas and regular amoebas are closed subsets of
$\RR_>^n$ and $\RR^n$, respectively. 
For a single Laurent polynomial $f$, write \defcolor{$\scrA(f)$} for the amoeba of the hypersurface 
\defcolor{$\calV(f)$} given  by $f$.
This is a \demph{hypersurface amoeba}.
Each component of the complement $\RR^n\smallsetminus\scrA(f)$ of a hypersurface amoeba is an open convex set~\cite[Cor.~6.1.6]{GKZ}.

The structure of $(\CC^\times)^n$ is controlled by its group $\ZZ^n\simeq \defcolor{N}:=\Hom(\CC^\times,(\CC^\times)^n)$ of
cocharacters and dual group of characters $\ZZ^n\simeq \defcolor{M}:=\Hom((\CC^\times)^n,\CC^\times)$.
For example, Laurent polynomials are linear combinations of characters.
Both $(\CC^\times)^n$ and $\RR^n\simeq N\otimes_\ZZ\RR$ have related structures.
A subtorus $\TT\subset(\CC^\times)^n$ corresponds to a saturated subgroup $\Pi\subset N$ ($\Pi$ is the set of
cocharacters of $\TT=\Pi\otimes_\ZZ\CC^\times$), as well as to a rational linear subspace
$\defcolor{\Pi_\RR}:=\Pi\otimes_\ZZ\RR$ of $\RR^n$, which is the amoeba of $\TT$.
All rational linear subspaces of $\RR^n$ arise in this manner.

A translate \defcolor{$a\TT$} of a subtorus $\TT$ by an element $a\in(\CC^\times)^n$ is an \demph{affine subtorus}.
A \demph{rational affine subspace} is the amoeba of an affine subtorus, equivalently, it is a translate of a rational linear
subspace.  
A subvariety $V\subset(\CC^\times)^n$ is \demph{degenerate} if it lies in a proper affine subtorus.
Otherwise it is \demph{nondegenerate}.
In Section~\ref{S:not_finite}, we observe that a variety $V$ is degenerate if and only if its amoeba lies in a proper
rational affine subspace of $\RR^n$.

The \demph{logarithmic limit set} $\defcolor{\scrL^\infty(V)}$ of a variety $V\subset(\CC^\times)^n$ is the set of
asymptotic directions of its amoeba, that is, the set of accumulation points of sequences 
$\bigl\{ \frac{z_m}{\|z_m\|}\bigr\}\subset S^{n-1}$ where  
$\{z_m\mid m\in\NN\}\subset\scrA(V)$ is unbounded.
Fundamental work of Bergman~\cite{Berg} and Bieri-Groves~\cite{BiGr} show that $\scrL^\infty(V)$ is the intersection of the
sphere $S^{n-1}$ with a rational polyhedral fan of pure dimension equal to the dimension of $V$, called the tropical
variety of $V$. 

\section{Amoebas and coamoebas as semi-algebraic sets}\label{S:SA}

Properties of semi-algebraic sets are developed throughout the books~\cite{BPR,BCR}.
A subset $X\subset \RR^m$ is \demph{semi-algebraic} if it is defined by finitely many algebraic equations and
algebraic inequalities.
By the Tarski-Seidenberg Theorem, the image of a semi-algebraic set under a polynomial map is a semi-algebraic set.

The unit complex numbers $\UU$ form a real algebraic variety as they are the fixed points of the anti-holomorphic involution
$z\mapsto\overline{z}^{-1}$ on $\CC^\times$.
This is realized concretely through the unit circle $c^2+s^2=1$ in $\RR^2$.
Indeed, $\CC[c,s]/\langle c^2+s^2-1\rangle \simeq \CC[x,y]/\langle xy-1\rangle$ under the map $x=c+s\sqrt{-1}$ and $y=c-s\sqrt{-1}$, and
$ \CC[x,y]/\langle xy-1\rangle \simeq \CC[x,x^{-1}]$ , which is the coordinate ring of $\CC^\times$.
  
The map $(\RR^\times)^n\times\UU^n \to (\CC^\times)^n$ defined by 
$  (\rho_1,\dotsc,\rho_n,\theta_1,\dotsc,\theta_n) \mapsto (\rho_1\theta_1,\dotsc,\rho_n\theta_n)$ is a $2^n$ to 1 cover, which is an
isomorphism on $(\RR_>)^n\times\UU^n$.
Given a complex algebraic subvariety $V\subset(\CC^\times)^n$, its pullback to $(\RR^\times)^n\times\UU^n$ is a real algebraic
subvariety.
The intersection of this pullback with $(\RR_>)^n\times\UU^n$ is then a semi-algebraic set that is homeomorphic to $V$.
Since the maps $|\cdot|$ and $\Arg$ on $(\RR_>)^n\times\UU^n$ are projection maps, the Tarski-Seidenberg Theorem implies the
following. 

\begin{proposition}
 The algebraic amoeba $|V|$ and the coamoeba $\coscrA(V)$ of an algebraic subvariety $V\subset(\CC^\times)^n$ are
 semi-algebraic subsets of $\RR_>^n$ and $\UU^n$, respectively.
\end{proposition}

While the amoeba $\scrA(V)$ of a variety is not semi-algebraic, it inherits finiteness and other properties of semi-algebraic sets from the
algebraic amoeba $|V|$.

This brings us to our first question.
We believe that the semi-algebraic nature of amoebas and coamoebas has been neglected.

\begin{question}\label{Q:semialgebraic}
 Given an algebraic subvariety $V\subset(\CC^\times)^n$, produce a description of its algebraic amoeba $|V|$ and coamoeba
 $\coscrA(V)$  as semi-algebraic subsets of $\RR_>^n$ and $\UU^n$.
\end{question}

Later in this section, we will give a description of the algebraic amoeba of a line in $(\CC^\times)^n$.
Besides amoebas of lines, the coamoebas described in~\cite[\S3]{NSnA}, and those coming from
discriminants~\cite{NS_linear,NP,PS}, 
we know of no other instances where such a semi-algebraic description has been given.
We expect that the description of the amoeba of the set of rectangular matrices that do not have full rank as a finite intersection of 
hypersurface amoebas given in~\cite{GG} will lead to a description of these amoebas as semi-algebraic sets.
Below, we show the algebraic amoeba of the parabola $y=(x{-}1)(x{-}2)$ and the hyperbola $y=1+1/(x{-}2)$.
\[
  \begin{picture}(150,105)(-14,0)
   \put(0,0){\includegraphics{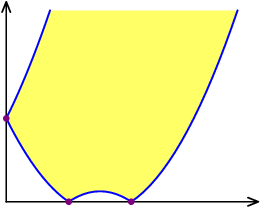}}
   \put(-14,82){$|y|$}   \put(105,10){$|x|$}
  \end{picture}
  \qquad\qquad
   \begin{picture}(127,105)
   \put(0,0){\includegraphics{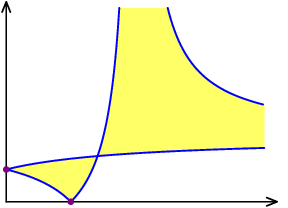}}
   \put(6,82){$|y|$}   \put(115,10){$|x|$}
  \end{picture}
\]
These have easy descriptions as semi-algebraic sets.  
Observe that  the boundaries of these amoebas are the absolute values of the real points of the curves.
This is not always the case.

The remainder of this section is devoted to providing a complete description of the algebraic amoeba of lines.
A \demph{line} is a subvariety $\ell\subset(\CC^\times)^n$ such that $\overline{\ell}\subset\PP^n$ is a linear subspace of dimension 1.
We first study lines in $(\CC^\times)^2$ and in $(\CC^\times)^3$, and then use those results to describe the algebraic amoeba of any
line.

\begin{example}\label{Ex:lineInPlane}
 We describe the algebraic amoeba of a nondegenerate line in $(\CC^\times)^2$.
 Let $a,b,c\in\CC^\times$.
 A point $(|x|,|y|)\in \RR^2_>$ lies in the algebraic amoeba of the line $ax+by+c=0$ if and only if there is a triangle with sides 
 $|a||x|$, $|b||y|$, and $|c|$.
 Equivalently, if $|x|,|y|>0$, 
 \begin{equation}\label{Eq:amoeba_line}
   |a||x|+|b||y|\ \geq\ |c|\,,\qquad\mbox{and}\qquad \bigl| |a||x|-|b||y|\bigr|\ \leq\ |c|\,.
 \end{equation}
 This is the shaded polyhedron shown in Figure~\ref{F:AA_Plane}.
 \hfill{$\diamond$}
\end{example}
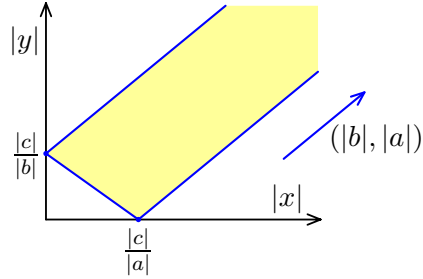
\begin{figure}[htb]
 \centering
   \begin{picture}(160,96)(-15,-10)
    \put(0,0){\includegraphics[height=86pt]{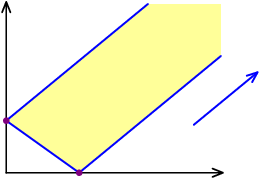}}
    \put(88,8){\small$|x|$}   \put(-11,68){\small$|y|$}
    \put(-10,24.5){$\frac{|c|}{|b|}$}   \put(32,-12){$\frac{|c|}{|a|}$}
    \put(110,31){\small$(|b|,|a|)$}
   \end{picture}
 \caption{Algebraic amoeba of a line in the plane.}\label{F:AA_Plane}
\end{figure}

The closure $\overline{\ell}$ of a nondegenerate line 
$\ell\subset(\CC^\times)^3\subset\PP^3$ meets each of the four coordinate planes in $\PP^3$ in a distinct point.
A line $\ell$ may be identified with $\CC\PP^1$.
A \demph{circle} on $\CC\PP^1$ is any image of $\RR\PP^1$ under a M\"obius transformation.

As explained in a discussion about coamoebas~\cite[\S3]{NS13}, 
if the four points of $\overline{\ell}\smallsetminus\ell$ lie on a circle, then after a reparameterization the line is
real (defined or parameterized by real polynomials), with points in the complement of the circle mapped two-to-one to the amoeba and
those on the circle mapped one-to-one to the relative boundary of the amoeba. 
If the four points do not lie on a circle, then the map from $\ell$ to its amoeba has no critical
points~\cite[Lem.~7]{NS13}.
This description is independent of parameterization of $\ell$.

\begin{lemma}\label{L:real_line}
  The algebraic amoeba $|\ell|$ of a nondegenerate real line $\ell\subset(\CC^\times)^3$ consists of the points on a symmetric quadratic
  hypersurface $Q\subset\RR^3$ of one sheet that lie in the positive octant,  $x,y,z>0$, and that also satisfy the linear triangle
  inequalities~\eqref{Eq:amoeba_line} arising from the projections of $\ell$ to each of the three coordinate planes.
  An explicit description of $Q$ is given in the proof below.
\end{lemma}
\begin{proof}
 We prove this by a direct calculation.
 We may assume that a nondegenerate real line $\ell\subset(\CC^\times)^3$ is given by a map
\[
   t\ \longmapsto\ (t, a_2(t-b_2), a_3(t-b_3))\,,
\] 
where $a_2,a_3,b_2,b_3\in\RR^\times$ with $b_2\neq b_3$.
If we set $t=p+q\sqrt{-1}$ for $(p,q)\in\RR^2$, then the algebraic amoeba is a dense subset of the image of
the map 
\[
  \RR^2\ni(p,q)\ \longmapsto\ 
  \Bigl(\,\sqrt{p^2+q^2}\,, \ 
     \sqrt{a_2^2(p^2+q^2-2b_2p+b_2^2)}\,, \ 
     \sqrt{a_3^2(p^2+q^2-2b_3p+b_3^2)}\, \Bigr)\ .
\]
 Letting $x,y,z$ be the coordinates for $(\RR_\geq)^3$ and squaring gives
\[
  x^2=p^2+q^2\,,\ \ 
  y^2=a_2^2(p^2+q^2-2b_2p+b_2^2)\,,\ \mbox{ and }\ 
  z^2=a_3^2(p^2+q^2-2b_3p+b_3^2)\,.
\]
 Eliminating $p$ and $q$ gives the quadratic equation
\[
  (b_2-b_3)x^2 \ +\ 
  \frac{b_3}{a_2^2}y^2 \ -\ \frac{b_2}{a_3^2}z^2 \ +\ 
  b_2b_3(b_2-b_3)\ =\ 0\,,
\]
 that is satisfied by points on the algebraic amoeba.
 It is symmetric about the coordinate planes, as it is linear in $x^2,y^2,z^2$.
 Since projection to any coordinate ($(x,y)$-, $(x,z)$-, or $(y,z)$-) plane in $(\CC^\times)^3$ gives a line, the corresponding 
 inequalities~\eqref{Eq:amoeba_line} must also be satisfied.
\end{proof}

\begin{example}\label{Ex:real_line}
 The algebraic amoeba of the real line with parameterization $t\mapsto(t,t{+}1,t{-}1)$ is displayed on the left in
 Figure~\ref{F:AA}. 
 It is the set of shaded points on the quadric $Q$ defined by $2x^2-y^2-z^2+2=0$, which are those points satisfying $x,y,z>0$, 
\[
   |x+y|\geq 1\,,\qquad\mbox{and}\quad |x-y|\leq 1\,.
\]
 These are the inequalities~\eqref{Eq:amoeba_line} from its projection to $(x,y)$-plane, which is the algebraic amoeba of
 the projection of $\ell$ to the corresponding coordinate $(\CC^\times)^2$.
 Interestingly, only one such coordinate projection is needed, which we invite the reader to check.
 (Note that the tangent to the quadric at $a=(1,2,0)$ is vertical, and that the line segments of its boundary lie along lines in the two
 rulings of $Q$.)
 \hfill$\diamond$
\end{example}

\begin{figure}[htb]
\centering
  \begin{picture}(115,110)(0,-7)
    \put(0,0){\includegraphics[height=100pt]{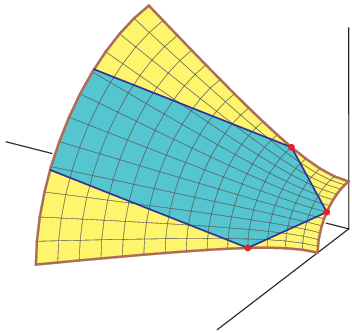}}
    \put(3,60){$x$}  \put(70,-2){$y$}    \put(96,84){$z$}
    \put(69,16){$a$}
   \end{picture}
\qquad
 \begin{picture}(112,110)(0,-5)
   \put(0,0){\includegraphics[height=100pt]{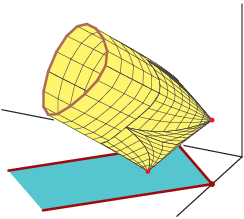}}
      \put(3,53){$x$}  \put(88,2){$y$}    \put(103,92){$z$}
 \end{picture}
\quad
  \begin{picture}(124,110)
   \put(0,0){\includegraphics[height=110pt]{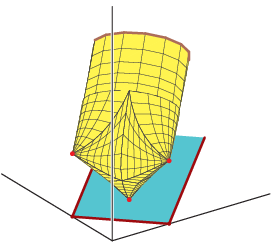}}
      \put(112,27){$x$}  \put(4,35){$y$}    \put(43,101){$z$}
  \end{picture}
 \caption{Algebraic amoebas of real and complex lines in $(\CC^\times)^3$.}
 \label{F:AA}
\end{figure}
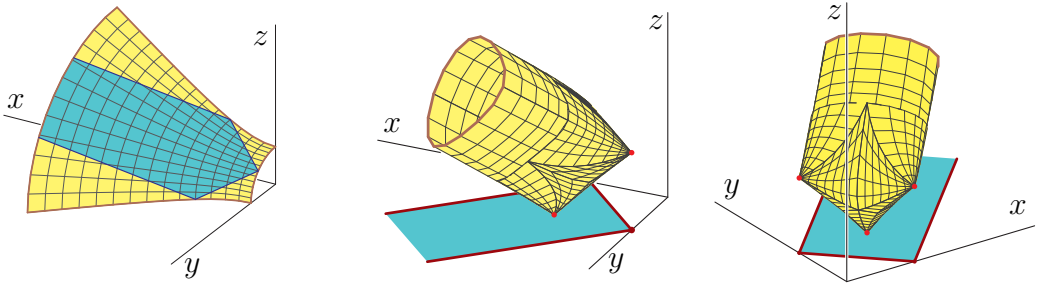

\begin{lemma}\label{L:complex_line}
  The algebraic amoeba $|\ell|$ of a nondegenerate complex line $\ell\subset(\CC^\times)^3$ consists of the points on a quartic hypersurface
  in the positive octant.
  An explicit equation for the quartic may be recovered from the proof below.
\end{lemma}
\begin{proof}
 Now consider a complex line $\ell\subset(\CC^\times)^3$.
 Its closure $\overline{\ell}$ in $\PP^3$ is identified with $\PP^1$.
 After a change of coordinates, we may assume that $\overline{\ell}\smallsetminus \ell=\{0,1,\infty, \alpha\}$, where $\alpha$ is not real.
 Then $\ell$ has a parameterization
 \begin{equation}\label{Eq:parameterization}
     t\ \longmapsto\  ( t, c(t-1), d(t-\alpha))\ ,
 \end{equation}
 for some $c,d\in\CC^\times$.
 Rescaling the last two coordinates (dividing by $c$ and $d$, respectively) 
 the parameterization for $\ell$ becomes $t\mapsto(t,t{-}1,t{-}\alpha)$.
 Writing $\alpha=-a{-}b\sqrt{-1}$ with $a,b\in\RR$ and $b\neq 0$, if $(x,y,z)$ is the point on $|\ell|$ corresponding to
 $t=p{+}q\sqrt{-1}$, then  
\[
   x^2\ =\ p^2+q^2\,,\ \ 
   y^2\ =\ (p-1)^2+q^2\,,\ \quad\mbox{ and }\ 
   z^2\ =\ (p+a)^2+(q+b)^2\,.
\]
 Then we have 
\[
  ay^2+z^2 \ -\ (a{+}1)x^2 -a^2{-}b^2{-}a\ =\ 2bq
  \quad\mbox{ and }\quad
  x^2-y^2+1\ =\ 2p\,.
\]
 It follows that 
 \begin{equation}\label{Eq:quartic}
  (ay^2+z^2  - (a{+}1)x^2 -a^2-b^2-a)^2\ +\ 
  b^2(x^2-y^2+1)^2\ -\ 4b^2x^2\ =\ 0\,.
 \end{equation}
 Thus $|\ell|$ lies on the part of the quartic surface $\defcolor{Q}\subset\RR^n$ defined by~\eqref{Eq:quartic} 
 in the positive orthant. 
 Write \defcolor{$Q^+$} for this positive part.

 We claim that $Q^+=|\ell|$.
 As~\eqref{Eq:quartic} contains only even powers of $z$,
 the projection of the quartic $Q$ to the $(x,y)$-plane factors through the map $(x,y,z)\mapsto (x,y,z^2)$.
 The image of $Q$ is a surface \defcolor{$S$} on which $\defcolor{w}:=z^2$ is a quadratic function of $x$ and $y$ with 
 discriminant 
\[
  \Delta\ =\ -4b^2(x-y-1)(x-y+1)(x+y-1)(x+y+1)\,.
\]
 In the quadrant where $x,y> 0$, this discriminant is nonnegative on the polyhedron $P$ defined by $|x+y|\geq 1$ and
 $|x-y|\leq 1$, which is the algebraic amoeba of the projection of $\ell$ to the $(x,y)$-plane
 (this projection is defined by $x-y-1=0$).
 Thus the surface $S$ has two branches above points in the interior of $P$ (where the discriminant is positive).
 Note that every point of $S$ with a positive third coordinate $w$ gives two points $\pm\sqrt{w}$ on $Q$ with exactly one
 on $Q^+$. 
 We claim that $w$ is nonnegative on both branches of $S$ and thus that $Q^+$ has two branches (and $Q$ has four branches)
 above $P$. 
 
 For this, we show that $|\ell|\subset Q^+$ has two points over all interior points of $P$, except $(|\alpha|,|\alpha-1|)$.
 Then the composition $\ell\to |\ell|\subset Q^+ \to P$, together with~\cite[Lem.~7]{NS13} which asserts that 
 $\ell \to |\ell|$ has no critical points, shows that $|\ell|=Q^+$.

 Indeed, let $t\in\CC\smallsetminus\RR$.  Then $t= p+q\sqrt{-1}$ with $q\neq 0$, and we have
 \begin{eqnarray*}
  |\ell(t)| &=&  \left(\sqrt{p^2+q^2}\,,\, \sqrt{(p-1)^2+q^2}\,,\, \sqrt{(a+p)^2+(b+ q)^2} \right)\\
  |\ell(\overline{t})| &=&  \left(\sqrt{p^2+q^2}\,,\, \sqrt{(p-1)^2+q^2}\,,\, \sqrt{(a+p)^2+(b- q)^2} \right)\,,
 \end{eqnarray*}
 As the first two coordinates of both coincide, these have the same image in the interior of $P$, but different third coordinates.
 When $t=\alpha$ or $t=\overline{\alpha}$, one of these has third coordinate $0$ and does not lie on $|\ell|$.
 (For $t=\alpha,\overline{\alpha}$, both expressions have the same image $(|\alpha|,|\alpha-1|)$.
 Thus $|\ell|=Q^+$.
\end{proof}

The equation~\eqref{Eq:quartic} describes the Zariski closure of $|\ell|$ when the last two coordinates in~\eqref{Eq:parameterization}
are rescaled.
If we did not rescale these coordinates, the equation~\eqref{Eq:quartic} becomes
\[
  (a\tfrac{1}{c\overline{c}}y^2+\tfrac{1}{d\overline{d}}z^2  - (a{+}1)x^2 -a^2-b^2-a)^2\ +\ 
  b^2(x^2-\tfrac{1}{c\overline{c}}y^2+1)^2\ -\ 4b^2x^2\ =\ 0\,.
\]

\begin{example}\label{Ex:complex_line}
 Let $\zeta$ be a primitive third root of unity.
 The algebraic amoeba of the symmetric complex line with parameterization $t\mapsto (t{-}1,t{-}\zeta,t{-}\zeta^2)$ is
\[ 
    x^4+y^4+z^4\,-\, (x^2y^2+x^2z^2+y^2z^2)\, -\, 3(x^2+y^2+z^2)\ +\ 9\ =\ 0\,.
\]
 Its closure meets the coordinate planes in the (singular) points 
 $(\sqrt{3},\sqrt{3},0)$,  $(\sqrt{3},0,\sqrt{3})$, and  $(0,\sqrt{3},\sqrt{3})$,
 and its projection to the $(x,y)$-plane is determined by the inequalities 
 $|x+y|\geq 2\sqrt{3}$ and $|x-y|\leq 2\sqrt{3}$.
 Two views of this algebraic amoeba and its projection to the $(x,y)$-plane are shown on the right in Figure~\ref{F:AA}. 
 \hfill{$\diamond$}
\end{example}

We now describe the algebraic amoeba of an arbitrary line $\ell$ in $(\CC^\times)^n$.
Let $\defcolor{\scrE_\ell}:=\overline{\ell}\smallsetminus\ell$ be the intersection of $\overline{\ell}$ with the
coordinate planes $\PP^n\smallsetminus(\CC^\times)^n$, this is its set of \demph{ends}.
Note that $|\scrE_\ell|\geq 2$.
Recall that a circle on $\PP^1$ is any image of $\RR\PP^1$ under a M\"obius transformation.
The line $\ell$ is \demph{real} if $\scrE_\ell\subset\overline{\ell}\simeq\PP^1$ lies on a circle and \demph{complex} if
$\scrE_\ell$ does not lie on a circle.

\begin{lemma}\label{L:linear}
 A line $\ell\subset(\CC^\times)^n$ lies on an affine subtorus $a\TT_\ell$ of dimension $|\scrE_\ell|{-}1$ whose closure is a
 linear subspace of\/ $\PP^n$.
\end{lemma}

\begin{proof}
 Since $\PP^n\smallsetminus(\CC^\times)^n$ consists of $n{+}1$ coordinate planes and $\ell$ meets each coordinate plane in a point of
 $\scrE_\ell$, any difference $n{+}1{-}|\scrE_\ell|$ comes from points of $\scrE_\ell$ lying on more than one coordinate plane.
 
 Suppose that a point of $\scrE_\ell$ lies in two coordinate planes of $\PP^n$, say $x=0$ and $y=0$.
 These coordinates are characters of $(\CC^\times)^n$ and they give a coordinate projection to $(\CC^\times)^2$.
 The image of $\ell$ in the $\CC^2$ containing this $(\CC^\times)^2$ is a line $\ell'$ passing through the origin.
 Thus $\ell'$ (and therefore $\ell$) satisfies an equation $y=ax$ for some $a\in\CC^\times$.

 Given a point of $\scrE_\ell$ lying on two or more coordinate planes, let $x_{i_0}=0, \dotsc, x_{i_r}=0$ be those
 coordinate planes.
 These characters $x_{i_j}$ give a coordinate projection to $(\CC^\times)^{r+1}$.
 The image of $\ell$, and thus $\ell$ itself, satisfies an equation $x_{i_j}=a_jx_{i_0}$ for some $a_j\in\CC^\times$,
 for each $j=1,\dotsc,r$.
 These equations for all ends $\scrE_\ell$ of $\ell$
 give $n{+}1{-}|\scrE_\ell|$ independent linear equations that define an affine subtorus $a\TT_\ell$ of dimension
 $|\scrE_\ell|{-}1$ that contains $\ell$.
\end{proof}

We describe the algebraic amoeba of a line $\ell$ as a semi-algebraic set.

\begin{theorem}
 Let $\ell\subset(\CC^\times)^n$ be a line with algebraic amoeba $|\ell|$  and $a\TT_\ell$ the affine subtorus
 containing $\ell$ of Lemma~\ref{L:linear}. 
 When $|\scrE_\ell|=2$, $\ell=a\TT_\ell$ and $|\ell|$ is a rational affine line.
 When $|\scrE_\ell|=3$, $\ell$ is a nondegenerate line in $a\TT_\ell\simeq(\CC^\times)^2$ and $|\ell|$ is as
 described in Example~\ref{Ex:lineInPlane}.

 Suppose that $|\scrE_\ell|\geq 4$.
 If $\ell$ is complex, then the map $\ell\to|\ell|$ is a bijection.
 If $\ell$ is real, then this map is injective on the circle containing $\scrE_\ell$ and two-to-one on its complement.

 When  $\ell$ is real, $|\ell|$ lies on a surface that is the intersection of $\binom{|\scrE_\ell|}{4}$
 quadratic hypersurfaces, one for each projection from $a\TT_\ell$ to a coordinate $(\CC^\times)^3$ and $|\ell|$ is the
 subset of that surface satisfying inequalities~\eqref{Eq:amoeba_line} from each projection to a coordinate
 $(\CC^\times)^2$. 

 When $|\ell|$ is complex, $|\ell|$ is the intersection of $A$ quadratic hypersurfaces and $B$ quartic
 hypersurfaces, where $A$ is the number of subsets of $\scrE_\ell$ of cardinality four that lie on a circle and $B$ is the
 number of those that do not lie on a circle.
\end{theorem}

The quadrics and quartics are described explicitly in Lemmas~\ref{L:real_line} and~\ref{L:complex_line}.

\begin{proof}
 By Lemma~\ref{L:linear}, if we choose an isomorphism $a\TT_\ell\simeq(\CC^\times)^{|\scrE_\ell|-1}$ and redefine $n$, we
 may assume that $|\scrE_\ell|=n{+}1$ and that $\ell$ meets each coordinate plane in distinct points.
 The conclusions for $|\scrE_\ell|<4$ are immediate.
 Suppose that $|\scrE_\ell|\geq 4$.
 Any equation or inequality satisfied by the image of $|\ell|$ under a projection to a coordinate subspace is
 satisfied by $|\ell|$.
 Thus the inequalities~\eqref{Eq:amoeba_line} obtained from projections to each coordinate $(\CC^\times)^2$ are valid on
 $|\ell|$ as are any equations coming from a projection to a coordinate $(\CC^\times)^3$.
  Lemmas~\ref{L:real_line} and~\ref{L:complex_line} show that these equations from each coordinate $(\CC^\times)^3$ are
 quadratic and quartic as the image of the line in that  $(\CC^\times)^3$ is real or complex, respectively.

 We prove the assertions about the degree of the map $\ell\to|\ell|$.
 If $\ell$ is complex, then it is complex in a projection to some coordinate $(\CC^\times)^3$.
 By Lemma~\ref{L:complex_line}, the map from $\ell$ to the algebraic amoeba of such a projection is one-to-one, 
 thus the map $\ell\to|\ell|$ is one-to-one.
 When $\ell$ is real,  
 we may assume that $\scrE_\ell\subset\RR\PP^1\subset\overline{\ell}$ and complex conjugation on
 $\CC\subset\PP^1=\overline{\ell}$ is the usual conjugation.
 Then a point and its conjugate both have the same absolute value, which shows that the map on $\ell\smallsetminus\RR\PP^1$
 is at least two-to-one.  
 The projection to a coordinate $(\CC^\times)^2$ is two-to-one on this set and one-to-one on the real points
 $\RR\PP^1\smallsetminus\scrE_\ell$.
 This proves the assertion about the degree of $\ell\mapsto|\ell|$ when $\ell$ is real.

 We show that the necessary inequalities and equations are sufficient to define $|\ell|$.
 In Lemma~\ref{L:real_line}, the $z$-coordinate of $|\ell|$ is a function of the $x$- and $y$-
 coordinates, as $|\ell|$ is a graph over its projection to the $(x,y)$-plane.
 Thus when $\ell$ is real, the points of $|\ell|$ are determined by the quadratic equations from these projections to
 each coordinate $(\CC^\times)^3$, and the inequalities from further projections to each coordinate $(\CC^\times)^2$.

 When $\ell$ is complex, at least one projection to a coordinate $(\CC^\times)^3$ is a complex line.
 As shown in the proof of Lemma~\ref{L:complex_line}, under the further projection to a coordinate $(\CC^\times)^2$, this is the
 graph of two functions, coming from the branches of a quadratic in $z^2$.
 When the projection to a coordinate $(\CC^\times)^3$ is real, the previous paragraph shows that the coordinate functions
 may be recovered from the inequalities and the quadratic equation for this projection.
 Thus the points of $|\ell|$ are determined by the quadratic and quartic
 equations coming from projections to each coordinate $(\CC^\times)^3$.
\end{proof}

\section{The dimension of an amoeba}\label{S:dimension}

The dimension of an amoeba may be understood in differential-geometric terms.
At a smooth point $x$ of a variety $V\subset(\CC^\times)^n$ of dimension $k$, the rank of the differential $d_x\Log$ of the
map to the amoeba is $2k{-}l$, where $l$ is the dimension as a real vector space of the intersection of $T_xV$ with the tangent space of the  
fiber $\UU^n$ at $x$. 
Since $\sqrt{-1}\cdot T_x\UU^n=T_x\RR_>^n$ and $V$ is complex, $d\Log$ and $d\Arg$ have the same rank on 
$T_xV$.  
Thus the amoeba and coamoeba of $V$ have the same dimension.
We would like to understand the dimension of $\scrA(V)$ from the geometry of $V$.

We have the bound $\dim_\RR\scrA(V)\leq\min\{n, 2\dim_\CC V\}$ as $\scrA(V)\subset\RR^n$ and
$\dim_\RR V=2\dim_\CC V$, and we seek structures on $V$ that imply this inequality is strict.
If a subvariety $V\subset(\CC^\times)^n$ has an action by a subtorus $\TT$ of dimension $l$, then
the orbit space $V/\TT$ is a subvariety of the quotient torus $(\CC^\times)^n/\TT$.
The amoeba $\scrA(V)\subset\RR^n$ has a translation action by the $l$-dimensional rational subspace $\scrA(\TT)$ with
orbit space $\scrA(V/\TT)$.
Taking this into account, we conclude that 
$\dim_\RR\scrA(V)\leq \min\{ n, 2\dim_\CC V{-}l\}$.

If $V$ lies in an affine subtorus $a\TT$, then its amoeba lies in $\Log(a)+\scrA(\TT)$, a rational affine subspace of the
same dimension as $\TT$.
This further bounds $\dim_\RR\scrA(V)$.

We identify a structure on $V$ that generalizes these observations.
We write $\dim X$ for the dimension of a complex variety $X$ and reserve $\dim_\RR$ for dimension as a real analytic set.

\begin{definition}\label{De:diminish}
 Let $V\subset(\CC^\times)^n$ be an irreducible subvariety and $\TT\subset(\CC^\times)^n$ a subtorus.
 We say that $\TT$ has a \demph{diminishing action}\footnote{This was called a near action in the original version of this paper.} on $V$ if
 we have 
 \begin{equation}\label{Eq:DA}
   \dim\TT \ <\  2(\dim V - \dim W)
   \quad\mbox{and}\quad
    2\dim W\ <\ n\ -\ \dim\TT \,,
 \end{equation}
 where $\defcolor{W}:=(\TT\cdot V)/\TT$ is the image of $V$ in the quotient torus $(\CC^\times)^n/\TT$.\hfill{$\diamond$}
\end{definition}

 A general fiber $F$ of the map $V\twoheadrightarrow W$ lies in an affine subtorus $a\TT$.
 The first inequality of~\eqref{Eq:DA} implies that $F$ has small codimension in $a\TT$ and the second inequality of~\eqref{Eq:DA} implies that
 $W$ has large codimension in $(\CC^\times)^n/\TT$.
 They together imply that
 \begin{equation}\label{Eq:single_condition}
   2\dim W\ +\ \dim\TT\ <\ \min\{n\,,\, 2\dim V\}\,.
 \end{equation}
%

\begin{example}
 We give three examples of varieties with a diminishing action by a subtorus.
 If a nontrivial proper torus $\TT$ acts on $V$ with $n > 2\dim V-\dim \TT$, then $\TT$ has a diminishing action on
 $V$, as $W=V/\TT$ has dimension $\dim V - \dim\TT$.

 If $V\subset(\CC^\times)^n$ lies in a proper affine subtorus $a\TT$, then 
 $W=(\TT\cdot V)/\TT$ is a point and has dimension zero.
 If $2\dim V > \dim\TT$ we have $\dim_\RR\scrA(V)\leq\dim\TT<\min\{n,2\dim V\}$ and so $\TT$ has a diminishing action on
 $V$. 

 Let $P\subset\TT\simeq(\CC^\times)^3$ be a hypersurface with a three-dimensional amoeba and
 $\ell\subset\TT'\simeq(\CC^\times)^3$ be a nondegenerate line.
 If we set $V:=P\times\ell\subset\TT\times\TT'$, then $\TT$ has a diminishing action on $V$ as
 in this case $W=\ell$ and $\dim V=\dim\TT=3$ but $\dim W=1$ and $n=6$ so that the inequalities in
 Definition~\ref{De:diminish} hold.
 Note that $\scrA(V)=\scrA(P)\times\scrA(\ell)$, so that $\dim_\RR\scrA(V)=5<\min\{n, 2\dim V\}$.
 \hfill{$\diamond$}
\end{example}

\begin{theorem}\label{Th:dimensionFirst}
 Let $V\subset(\CC^\times)^n$ be an irreducible subvariety.
 If a nontrivial proper subtorus $\TT$ has a diminishing action on $V$, then
 $\dim_\RR\scrA(V)<\min\{n,2\dim V\}$. 
\end{theorem}

\begin{proof}
 Let $\TT\subset(\CC^\times)^n$ be a nontrivial a proper subtorus and let $W=\TT\cdot V/\TT\subset(\CC^\times)^n/\TT$.
 Any fiber \defcolor{$F$} of $\scrA(V) \twoheadrightarrow\scrA(W)$ lies in
 a translation of $\scrA(\TT)$ and thus $\dim_\RR F\leq \dim \TT$.
 Suppose that $F$ is a general fiber.
 Then
 \[
   \dim_\RR \scrA(V)\  =\  \dim_\RR \scrA(W)\ +\ \dim_\RR F\ \leq\
   2\dim W + \dim \TT\,.
 \]
 If $\TT$ has a diminishing action on $V$, then $\dim_\RR \scrA(V)<\min\{n,2\dim_\CC V\}$, by~\eqref{Eq:single_condition}.
\end{proof}

We believe the following is true.

\begin{conjecture}\label{C:dim}
 For an irreducible subvariety $V\subset(\CC^\times)^n$, if $\dim_\RR\scrA(V)<\min\{n,2\dim V\}$, then there is a 
 nontrivial proper subtorus $\TT$ of $(\CC^\times)^n$ having a diminishing action on $V$.
\end{conjecture}

\begin{remark}
 After posing Conjecture~\ref{C:dim}, Draisma, Rau, and Yuen~\cite{DRY} gave the following formula for the dimension of an amoeba, which
 they showed implies Conjecture~\ref{C:dim}:
 \[
   \dim_\RR\scrA(V)\ =\
   \min\{2\dim \overline{\TT\cdot V}-\dim\TT \mid \TT\mbox{ is a subtorus of }(\CC^\times)^n\}\,.
 \]
 This result suggests the problem of giving explicit methods to determine this dimension.
 For example, what is the dimension of the amoeba of a linear subspace in $\CC^n$?
 \hfill{$\diamond$}
\end{remark}

We prove Conjecture~\ref{C:dim} when the dimension of the amoeba is the minimum possible.

\begin{theorem}
 Let $V\subset(\CC^\times)^n$ be an irreducible subvariety.
 If $V$ and its amoeba have the same dimension, then $V$ is an affine subtorus of $(\CC^\times)^n$.
\end{theorem}

\begin{proof}
 Suppose that $V$ is a hypersurface, so that $\dim V = n{-}1$.
 Then $\dim_\RR\scrA(V)=n{-}1$, so that it is a hypersurface in $\RR^n$.
 Since each component of $\RR^{n}\smallsetminus\scrA(V)$ is convex, $\scrA(V)$ must be a hyperplane, as it bounds every
 such component.
 Since the logarithmic limit set of $V$---the set of asymptotic directions of $\scrA(V)$---is a rational
 polyhedron in $S^{n-1}$  of dimension $n{-}2$, $\scrA(V)$ is a rational affine hyperplane, $a+\Pi_\RR$, for some
 subgroup $\defcolor{\Pi}\subset\ZZ^n$ of rank $n{-}1$.

 Let $\TT:=\Pi\otimes_\ZZ\CC^\times$ be the corresponding subtorus and let $\defcolor{W}:=(\TT\cdot V)/\TT$ be the image
 of $V$ in $\CC^\times\simeq(\CC^\times)^{n}/\TT$.
 Then $\scrA(W)$ is the image of $\scrA(V)$ in $\RR\simeq\RR^{n}/\Pi_\RR$.
 As this is a point and $W$ is irreducible, we conclude that $W$ is a point and $V$ is a single orbit of $\TT$.

 Now suppose that $V$ is not a hypersurface and set $\defcolor{k}:=\dim_\RR\scrA(V)=\dim V$.
 For every surjective homomorphism $\varphi\colon(\CC^\times)^n\twoheadrightarrow(\CC^\times)^{k+1}$ with $\varphi(V)$ a
 hypersurface, $\scrA(\varphi(V))=\Phi(\scrA(V))$, where $\Phi$ is the corresponding linear surjection
 $\RR^n\twoheadrightarrow\RR^{k+1}$. 
 By the previous arguments, $\varphi(V)$ is an affine subtorus of $(\CC^\times)^{k+1}$, and therefore $V$ lies in
 an affine subtorus of dimension $n{-}1$.
 Doing this for sufficiently many independent homomorphisms $\varphi$ and taking the intersections of the affine subtori of 
 dimension $n{-}1$ proves the theorem.
\end{proof}

\section{Most amoebas do not have a finite amoeba basis}\label{S:not_finite}

As introduced by Schroeter and de Wolff~\cite{SdW}, a subvariety $V\subset(\CC^\times)^n$ has a 
\demph{finite amoeba basis} if there exist Laurent polynomials $f_1,\dotsc,f_r$ such that 
\[
  \scrA(V)\ =\ \scrA(f_1)\cap \scrA(f_2)\cap \dotsb \cap \scrA(f_r)\,.
\]
Theorem~\ref{T:noFAB} shows that an irreducible nondegenerate variety whose amoeba has dimension less than $n$ does not
have a finite amoeba basis, and we expect that it is rare for a variety to have a finite amoeba basis.
We also exhibit some varieties with a finite amoeba basis.

By the distributivity of union over intersection, we have the following lemma.

\begin{lemma}\label{L:FAB}
 The class of varieties admitting a finite amoeba basis is closed under finite union.
\end{lemma}

The \demph{Newton polytope} $\defcolor{P(f)}$ of a Laurent polynomial $f$ is the convex hull of the set of exponents of its
non-zero monomials. 
A variety $V\subset(\CC^\times)^n$ is an \demph{independent complete intersection} if it is the set-theoretic complete 
intersection of polynomials $f_1,\dotsc,f_r$ whose Newton polytopes are affinely independent.
Any affine subtorus is an independent complete intersection, as subtori of codimension $r$ are defined by $r$
independent binomials.

\begin{theorem}\label{T:FAB}
  An independent complete intersection $V$ has a finite amoeba basis.
  If $f_1,\dotsc,f_r$ are polynomials defining $V$ set-theoretically with affinely independent Newton polytopes, then 
\[
     \scrA(V)\ =\ \scrA(f_1)\cap \scrA(f_2)\cap \dotsb \cap \scrA(f_r)\,.
\]
\end{theorem}

\begin{proof}
 For each $i=1,\dotsc,r$, multiply $f_i$ by a monomial so that its Newton polytope $P(f_i)$ contains the origin.
 Let \defcolor{$M_i$} be the saturated sublattice of the character lattice $M$ spanned by $P(f_i)$
 and suppose that  $a_i$ is the rank of $M_i$.
 Let \defcolor{$x^{(i)}_1,\dotsc,x^{(i)}_{a_i}$} be independent characters that generate $M_i$.
 Then $f_i$ is a Laurent polynomial in these characters.
 That is, there is a Laurent polynomial \defcolor{$\overline{f_i}$} in $a_i$ variables
 \defcolor{$y^{(i)}_1,\dotsc,y^{(i)}_{a_i}$}, such that for 
 $x\in(\CC^\times)^n$, we have $f_i(x)=\overline{f_i}(x^{(i)}_1,\dotsc,x^{(i)}_{a_i})$.
 Write \defcolor{$\bx^{(i)}$} for this list $(x^{(i)}_1,\dotsc,x^{(i)}_{a_i})$ of characters.

 Consider the map 
\[
   \defcolor{\varphi}\ \colon\ (\CC^\times)^n\ \longrightarrow\ 
    (\CC^\times)^{a_1}\times (\CC^\times)^{a_2}\times \dotsb  \times(\CC^\times)^{a_r}
\]
 defined by $x\mapsto(\bx^{(1)},\bx^{(2)},\dotsc,\bx^{(r)})$.
 As  $f_i(x)=\overline{f_i}(x^{(i)}_1,\dotsc,x^{(i)}_{a_i})$, we observe that 
 \begin{equation}\label{Eq:PullBackProduct}
   V\ =\ \varphi^{-1}(\calV(\overline{f_1})\times\dotsb\times\calV(\overline{f_r}))\,,
 \end{equation}
 the product of the hypersurfaces
 $\calV(\overline{f_i})$ in $(\CC^\times)^{a_i}$ for $i=1,\dotsc,r$.
 
 By our assumption on the Newton polytopes, these characters 
 $\bx^{(1)},\bx^{(2)},\dotsc,\bx^{(r)}$ are independent.
 As the ideal of the image of $\varphi$ is generated by binomials which arise from integer linear relations among these
 characters~\cite[Ch.~4]{GBCP}, $\varphi$ is surjective.

 Under the map $\Log$, $\varphi$ induces a surjective linear map
\[
   \defcolor{\Phi}\ \colon\ \RR^n\ \longtwoheadrightarrow\ \RR^{a_1}\times \RR^{a_2}\times \dotsb \times \RR^{a_r}\,.
\]
 We claim that $\scrA(V)=\Phi^{-1}(\scrA(\overline{f_1})\times\dotsb\times\scrA(\overline{f_r}))$.
 That $\scrA(V)$ is contained in the inverse image of the product of hypersurface amoebas is a consequence
 of~\eqref{Eq:PullBackProduct}.
 For the other direction, suppose that $z=(z^{(1)},\dotsc,z^{(r)})$ with $z^{(i)}\in\RR^{a_i}$ for $i=1,\dotsc,r$ is a
 point in the product of the hypersurface amoebas. 
 For each $i=1,\dotsc,r$ let $\defcolor{y^{(i)}}\in(\CC^\times)^{a_i}$  a point in the hypersurface $\calV(\overline{f_i})$.
 Then $\defcolor{y}:=(y^{(1)},\dotsc,y^{(r)})$ lies in the product of the hypersurfaces.

 As $\varphi$ is surjective, there is a point $x\in(\CC^\times)^n$ with $\varphi(x)=y$.
 By~\eqref{Eq:PullBackProduct},  $\varphi^{-1}(y)\subset V$.  
 Applying $\Log$ shows that $\Phi^{-1}(z)\subset\scrA(V)$.
 This implies the other containment, so that 
 $\scrA(V) = \Phi^{-1} \bigl(\scrA(\overline{f_1})\times\dotsb\times\scrA(\overline{f_r}) \bigr)$.
 Since 
\[
   \scrA(f_i)\ =\ 
    \Phi^{-1}(\RR^{a_1}\times\dotsb\times\RR^{a_{i-1}}\times\scrA(f_i)\times\RR^{a_{i+1}}\times\dotsb\times\RR^{a_r})\,,
\]
 the conclusion of the theorem holds.
\end{proof}

The original version of this paper conjectured that finite unions of independent complete intersections were the only varieties with a
finite amoeba basis.
Goucha and Gouveia~\cite{GG} showed this was too optimistic.

\begin{proposition}[Corollary 6.3 of~\cite{GG}]
  For positive integers  $m,n$, let $\defcolor{M_{m,n}}=(\CC^\times)^{mn}$ be the set of $m\times n$ matrices with nonzero complex entries.
  The set $X_{m,n}\subset M_{m,n}$ of matrices of rank strictly less than $\min\{m,n\}$ has a finite amoeba basis given by the maximal
  minors of matrices in $M_{m,n}$.  
\end{proposition}

While the set $X_{2,n}$ is a subtorus of $(\CC^\times)^{2n}$, and thus an independent complete intersection, when $3\leq m<n$, $X_{m,n}$ is
not a set-theoretic complete intersection of any subset of maximal minors.
For example, $X_{3,4}$ is one of two ten-dimensional components of the variety defined by any two maximal minors of the generic $3\times 4$
matrix $M_{3,4}$.

Irreducible nondegenerate independent complete intersections in $(\CC^\times)^n$ have amoebas of dimension $n$, and the amoeba of $X_{m,n}$
is also full-dimensional in $\RR^{mn}$.
This is no coincidence, as we show in the following result.

\begin{theorem}\label{T:noFAB}
  A nondegenerate irreducible variety $V\subset(\CC^\times)^n$ with an amoeba of dimension
  less than $n$ does not have a finite amoeba basis.
\end{theorem}

For example, a nondegenerate irreducible curve in $(\CC^\times)^n$ for $n\geq 3$ does not have a finite amoeba basis.
Theorem~\ref{T:noFAB} is a consequence of the following lemma.

\begin{lemma}\label{L:noFAB}
  Let $V\subset(\CC^\times)^n$ be a variety with amoeba of dimension $d<n$.
  If $V$ has a finite amoeba basis, then each component of $V$ lies in an affine subtorus of dimension at most $d$.
\end{lemma}

Our proof of Lemma~\ref{L:noFAB} uses the convexity of each component of the complement of a hypersurface amoeba, along with stratifications
of amoebas induced by stratifications of the corresponding algebraic amoebas given by their structure as semi-algebraic sets.
Let us first recall some results about semi-algebraic sets, these are discussed in Chapters 2 and 9 of~\cite{BCR}.

The intersection or union of finitely many semi-algebraic sets is semi-algebraic.
A semi-algebraic set $X\subset\RR^n$ has a decomposition into pairwise disjoint semi-algebraic subsets,
 \begin{equation}\label{Eq:SAD}
   X\ =\ X_0 \;\sqcup\; X_1 \;\sqcup\; \dotsb \;\sqcup\; X_n\,,
 \end{equation}
where each $X_i$ is a locally closed $C^\infty$ manifold of dimension $i$.
The maximum $d$ with $X_d\neq\emptyset$ is the dimension of $X$. 
This decomposition may be chosen so that each connected component of $X_i$ is homeomorphic to the product $(0,1)^i$ of open intervals,
which we call a \demph{box}.

The local dimension \defcolor{$\dim_xX$} of a point $x\in X$ is the maximum $i$ such that $x\in\overline{X_i}$.
Let \defcolor{$X_{\rm sm}$} be the (open) subset consisting of points $x\in X$ with a neighborhood in $X$ that is a manifold.
It is the set of smooth points of $X$ and its complement $\defcolor{X_{\rm sg}}:=X\smallsetminus X_{\rm sm}$ is the singular locus of $X$.
We have  $\overline{X_{\rm sm}}=X$  and $\dim X_{\rm sg}<\dim X$.
Both $X_{\rm sm}$ and $X_{\rm sg}$ are semi-algebraic, and we may assume that (the connected components of) any decomposition of $X$ into
semi-algebraic sets refines the decomposition $X=X_{\rm sm}\sqcup X_{\rm sg}$.

For $x\in\RR^n$ and $\epsilon>0$, let \defcolor{$B(x,\epsilon)$} be the open ball in $\RR^n$ centered at $x$ of radius $\epsilon$.
For a submanifold $M\subset\RR^n$ and $x\in M$, let $T_xM$ be the tangent space of $M$ at $x$, which we regard as an affine subspace of
$\RR^n$ containing $x$. 

\begin{proof}[Proof of Lemma~\ref{L:noFAB}]
  Suppose that $V$ has a finite amoeba basis given by Laurent polynomials $f_1,\dotsc,f_r$.
  Then 
 \begin{equation}\label{Eq:FiniteIntersection}
   \scrA(V)\ =\ \scrA(f_1)\;\cap\;\scrA(f_2)\;\cap\;\dotsb\cap\;\scrA(f_r)\,.
 \end{equation}
For each $i$, the decomposition~\eqref{Eq:SAD} of the algebraic amoeba $|\calV(f_i)|$ induces a decomposition
of the hypersurface amoeba $\scrA(f_i)$, 
 \begin{equation}\label{Eq:AmoebaDecomp}
     \scrA(f_i)\ =\ \scrA(f_i)_0\;\sqcup\;\scrA(f_i)_1\;\sqcup\;\dotsb\;\sqcup\;\scrA(f_i)_n\,,
 \end{equation}
where $\scrA(f_i)_j$ a locally closed $C^\infty$ manifold of dimension $j$.

For a smooth point $x\in\scrA(V)_{\rm sm}$, let $\defcolor{J_x}=(j_1,\dotsc,j_r)$ be defined by where $x$ lies in the
decompositions~\eqref{Eq:AmoebaDecomp}, that is, $x\in\scrA(f_i)_{j_i}$ for $i=1,\dotsc,r$.
For a list $J=(j_1,\dotsc,j_r)$, let
\[
   \defcolor{U_J}\ :=\ 
   \{x\in\scrA(V)_{\rm sm} \mid J_x=J\}\ =\ 
   \scrA(V)_{\rm sm}\cap\bigcap_{i=1}^r \scrA(f_i)_{j_i}\ .
\]
This is the image of a semi-algebraic subset of $|V|$, and may be further decomposed into boxes.
The collection \defcolor{$\calU$} of these for all $J$ give a finite decomposition of $\scrA(V)_{\rm sm}$ into boxes.

Let $y\in\scrA(V)_{\rm sm}$.
Then there is a box $U\in \calU$ of dimension $\dim_y\scrA(V)$ with $y\in\overline{U}$.
Let $x\in U$ and set $J:=J_x$, so that $U\subset U_J$.
Let $i\in\{1,\dotsc,r\}$.
Let \defcolor{$C$} be a connected component of $\RR^n\smallsetminus\scrA(f_i)$.
This is an open convex subset of $\RR^n$ disjoint from the manifold $\scrA(f_i)_{j_i}$.
As $x\in\scrA(f_i)_{j_i}$, we claim that there is some $\epsilon>0$ such that $B(x,\epsilon)\cap T_x\scrA(f_i)_{j_i}$ is disjoint from
$C$.
Indeed, if $x\not\in\overline{C}$, then there is a ball $B(x,\epsilon)$ disjoint from $\overline{C}$, and hence $C$.
If $x\in\overline{C}$, then the manifold $\scrA(f_i)_{j_i}$  lies in the boundary of $C$.
Let $H$ be a supporting hyperplane to $C$ through $x$.
Then $C$, and hence $\scrA(f_i)_{j_i}$, lies on one side of $H$, which implies that $H$ contains the tangent space
$T_x\scrA(f_i)_{j_i}$.
As $H$ is disjoint from the open convex set $C$, we may chose any $\epsilon>0$ in this case.

As there are only finitely many components in the complement of $\scrA(f_i)$, we may assume that $\epsilon$ has been chosen so that
$B(x,\epsilon)\cap T_x\scrA(f_i)_{j_i}$ is disjoint from every component, that is, 
 \begin{equation}\label{Eq:dsj}
   B(x,\epsilon)\cap T_x\scrA(f_i)_{j_i} \ \subset\ \scrA(f_i)\,.
 \end{equation}

Shrinking $\epsilon$ if necessary, we may assume that~\eqref{Eq:dsj} holds for all $i=1,\dotsc,r$, and also that
$B(x,\epsilon)\cap\scrA(V)\subset U$ (as $U$ is disjoint from the closed set $\scrA(V)\smallsetminus U$).
The finite intersection~\eqref{Eq:FiniteIntersection} implies that
\[
   B(x,\epsilon)\cap\scrA(V)\ =\
   B(x,\epsilon)\cap \scrA(f_1)\;\cap\;\scrA(f_2)\;\cap\;\dotsb\cap\;\scrA(f_r)\ \subset\ U\,.
\]
As for each $i=1,\dotsc,r$, $U\subset\scrA(f_i)_{j_i}$, we have the inclusion of tangent spaces
$T_xU\subset T_x\scrA(f_i)_{j_i}$ and thus
\[
    B(x,\epsilon)\cap T_xU\ \subset\  B(x,\epsilon)\cap T_x\scrA(f_i)_{j_i}\ \subset\
    B(x,\epsilon)\cap \scrA(f_i)\,.
\]    
This implies that $B(x,\epsilon)\cap U=B(x,\epsilon)\cap T_xU$ as both are manifolds of the same dimension that contain $x$.
Thus $U$ is an open subset of $T_xU$.
As $y\in\overline{U}$ is a smooth point of $\scrA(V)$, $y$ has a neighborhood in  $\scrA(V)$ that is an open subset of the affine space
$T_xU$.

Let $W\subset V$ be an irreducible component.
Suppose that $\scrA(W)$ is not a subset of the closure of the difference $\scrA(V)\smallsetminus\scrA(W)$.
If $y$ is a smooth point of  $\scrA(W)$  not lying in this closure, then $y$ is a smooth point of $\scrA(V)$.
We have just shown that $y$ has a neighborhood in $\scrA(W)$ that is an open subset of an affine space.
By Lemma~\ref{L:almost_linear} below, $W$ lies in an affine subtorus of dimension equal to
$\dim_y\scrA(W)\leq\dim\scrA(V)=d<n$.

Let $U$ be the union of all components $W$ of $V$ such that $\scrA(W)$ is not a subset of the closure of the difference
$\scrA(V)\smallsetminus\scrA(W)$.
Then the $\scrA(U)=\scrA(V)$ and every irreducible component of $U$ lies in an affine subtorus of dimension at most $d=\dim\scrA(V)$.

Let $W$ be a component of $V$ that is not a component of $U$.
Since $\scrA(W)\subset\scrA(V)$, it lies in a union of
finitely many rational affine subspaces of dimension at most $\dim\scrA(V)$ 
The irreducibility of $W$ implies that $\scrA(W)$ lies in one such rational affine subspace and 
 by Lemma~\ref{L:Affine=degenerate}, $W$ lies in an affine subtorus of dimension at most $d<n$.
\end{proof}

\begin{lemma}\label{L:almost_linear}
 Let $V\subset(\CC^\times)^n$ be an irreducible subvariety such that $\scrA(V)$ has a point $x$ having a neighborhood in 
 $\scrA(V)$ that is an open subset of a $d$-dimensional plane.  
 Then $V$ lies in a $d$-dimensional affine subtorus.
\end{lemma}

\begin{proof}
 Translating $V$ (and thus $\scrA(V)$) if necessary, we may assume that  $x=0$.
 There is a rational $d$-dimensional plane $\Pi_\RR$ such that the projection of $T_x\scrA(V)$ to $\Pi_\RR$ is surjective.
 Taking a subtorus complementary to $\Pi\otimes_\ZZ\CC^\times\simeq(\CC^\times)^d$, we have 
 coordinates $(\CC^\times)^n=(\CC^\times)^d\times(\CC^\times)^{n-d}$ and a decomposition of
 $\RR^n=\RR^d\oplus \RR^{n-d}$ into rational linear subspaces such that  $T_x\scrA(V)$ is the graph of a map
 $\Lambda\colon\RR^d\to\RR^{n-d}$ (here, $\Pi_\RR=\RR^d$). 
 That is, points of $T_x\scrA(V)$ are of the form
 \begin{equation}\label{Eq:graph}
   \{(y_1,\dotsc,y_d,\ell_1(y),\dotsc,\ell_{n-d}(y)) \mid y\in\RR^d\}\,,
 \end{equation}
 where $\ell_i$ are the coordinate functions of $\Lambda$, which are linear forms.
 By our assumption, $\scrA(V)$ agrees with $T_x\scrA(V)$ in a neighborhood of $x$, 
 so there is a neighborhood $U$ of the origin in $\RR^d$ such that this set~\eqref{Eq:graph} restricted to $y\in U$
 lies in $\scrA(V)$.  
 
 The exponential map on $\RR^n$ sends $\scrA(V)$ to the algebraic amoeba $\defcolor{|V|}$.
 Thus the set
 \begin{equation}\label{Eq:expGraph}
   \{ (e^{y_1},\dotsc, e^{y_d}\,,\, e^{\ell_1(y)},\dotsc,e^{\ell_{n-d}(y)})
     \mid y\in U\}
 \end{equation}
 is a neighborhood of the point $1$ in $|V|$.
 In particular, it is a semi-algebraic set.

 If we set $\defcolor{z_i}:=e^{y_i}$, then the exponential of a linear form becomes
\[
    e^{\ell_i(y)}\ =\ z_1^{\alpha_{i,1}}\dotsb z_d^{\alpha_{i,d}}\ =:\ \defcolor{z^{\alpha_i}}\,,
\]
 where $\ell_i(y)=\alpha_{i,1}y_1+\dotsb+\alpha_{i,d}y_d=\alpha_i\cdot y$.
 In particular, each monomial $z^{\alpha_i}$ is an algebraic function of $z_1,\dotsc,z_d$.
 This implies that the coefficients/exponents $\alpha_{i,j}$ are rational numbers.
 If we let \defcolor{$\delta$} be their common denominator and set $t_i:=z_i^{1/\delta}$ for $i=1,\dotsc,d$---this is
 well-defined as each $z_i>0$---then we may assume that each $\alpha_{i,j}$ is an integer.

 Let $\TT$ be the $d$-dimensional subtorus of $(\CC^\times)^n$ whose algebraic amoeba is the Zariski closure of the
 set~\eqref{Eq:expGraph}. 
 That is, $\TT$ is defined in $(\CC^\times)^n$ by $x_{n+i}=x_1^{\alpha_{i,1}}\dotsb x_d^{\alpha_{i,d}}$ for each
 $i=1,\dotsc, n{-}d$.
 Then the image of $V$ in $(\CC^\times)^n/\TT$  is contained in the compact subtorus of $(\CC^\times)^n/\TT$, which implies
 that the image of $V$ is a single point as $V$ is irreducible.
 This completes the proof.
\end{proof}

We close with a simple characterization of degenerate varieties.

\begin{lemma}\label{L:Affine=degenerate}
 An irreducible subvariety $V$ of $(\CC^\times)^n$ lies in an affine subtorus of dimension $d$ if and only is its amoeba
 lies in a rational affine subspace of $\RR^n$ of dimension $d$.
\end{lemma}

\begin{proof}
 Suppose that $\scrA(V)$ lies in a proper affine subspace $a\scrA(\TT)$, for a subtorus $\TT$ of $(\CC^\times)^n$. 
 Then the amoeba of the irreducible variety $(\TT\cdot V)/\TT$ is a point, which implies that $\TT\cdot V$  is a single
 orbit of $\TT$.
 Noting that $V\subset a\TT$ implies that $\scrA(V)\subset \Log(a)+\scrA(\TT)$, which is a rational affine subspace of $\RR^n$ 
 of dimension $\dim\TT$, completes the proof. 
\end{proof}

\providecommand{\bysame}{\leavevmode\hbox to3em{\hrulefill}\thinspace}
\providecommand{\MR}{\relax\ifhmode\unskip\space\fi MR }
\providecommand{\MRhref}[2]{%
  \href{http://www.ams.org/mathscinet-getitem?mr=#1}{#2}
}
\providecommand{\href}[2]{#2}

\end{document}